\documentclass[12pt, reqno]{amsart}
\usepackage{amsmath, amsthm, amscd, amsfonts, amssymb, graphicx, color}
\usepackage[bookmarksnumbered, colorlinks, plainpages]{hyperref}
\usepackage{amscd,graphicx,epsfig}
\usepackage[english]{babel}
\usepackage[utf8]{inputenc}

\title[Gromov's ellipticity of cubic hypersurfaces]{Algebraic Gromov's ellipticity 
of cubic hypersurfaces}
\author{Shulim Kaliman and Mikhail Zaidenberg}

\address{University of Miami, Department of Mathematics, Coral Gables, FL
33124, USA}
\email{kaliman@math.miami.edu}
\address{Univ. Grenoble Alpes, CNRS, IF, 38000 Grenoble, France}
\email{mikhail.zaidenberg@univ-grenoble-alpes.fr}


\newtheorem{theorem}{Theorem}[section]
\newtheorem*{theorem*}{Theorem}
\newtheorem*{conjecture*}{Conjecture}

\newtheorem{proposition}[theorem]{Proposition}
\newtheorem{lemma}[theorem]{Lemma}
\newtheorem{corollary}[theorem]{Corollary}
\newtheorem*{corollary*}{Corollary}

\newtheorem{question}[theorem]{Question}

\theoremstyle{definition}

\newtheorem*{example*}{Example}

\newtheorem{notation}[theorem]{Notation}

\theoremstyle{remark}
\newtheorem*{remark*}{Remark}
\newtheorem{remark}[theorem]{Remark}

\theoremstyle{remark}
\newtheorem*{remarks*}{Remarks}
\newtheorem{remarks}[theorem]{Remarks}

\newcommand{\PP}{{\mathbb P}}
\newcommand{\CC}{{\mathbb C}}

\newcommand{\QQ}{{\mathbb Q}}
\newcommand{\kk}{{\mathbb K}}

\def\cC{{\mathcal C}}
\def\cO{{\mathcal O}}
\newcommand{\A}{{\mathbb A}}

\newcommand{\ee}{\end{enumerate}}

\newcommand{\T}{\mathbb T}

\DeclareMathOperator{\End}{End}
\DeclareMathOperator{\id}{id}

\DeclareMathOperator{\Aut}{Aut}
\DeclareMathOperator{\SAut}{SAut}

\DeclareMathOperator{\SL}{SL}

\DeclareMathOperator{\Spec}{Spec}

\renewcommand{\subset}{\subseteq}

\renewcommand{\phi}{\varphi}

\begin{document}
\begin{abstract}
We show that every smooth cubic hypersurface 
$X$ in $\PP^{n+1}$, $n\ge 2$ is 
algebraically elliptic in Gromov's sense. 
This gives the first examples of non-rational 
projective manifolds elliptic in Gromov's sense.
We also deduce that the punctured affine cone 
over $X$ is elliptic. 
\end{abstract}
\maketitle

{\footnotesize \tableofcontents}

\section{Introduction}
Gromov's ellipticity appeared 
(and was extremely useful) 
in complex analysis
within the Oka-Grauert theory, see \cite{Gr}, 
\cite{For17a} and  \cite{For23}. 
In the same paper \cite{Gr} Gromov
developed an algebraic version of this notion. 
We deal below 
exceptionally with algebraic ellipticity. 
Thus, all varieties and vector bundles 
in this paper are algebraic, and `ellipticity' 
refers to algebraic ellipticity.

Let $\kk$ be an algebraically closed field  
of characteristic zero and $\A^n$ resp. $\PP^n$ 
be the affine resp. projective 
$n$-space over $\kk$. 

Given  a smooth algebraic variety $X$,
a \emph{spray of rank $r$} 
over $X$ is a triple $(E,p,s)$ where
$p\colon E\to X$ is  a vector bundle 
of rank $r$ with zero section $Z$ 
and $s\colon E\to X$ is a morphism  
such that $s|_Z=p|_Z$. A spray $(E,p,s)$ is 
\emph{dominating at $x\in X$} if the restriction 
$s|_{E_x}\colon E_x\to X$ 
to the fiber $E_x=p^{-1}(x)$ is dominant 
at the origin $0_x\in E_x\cap Z$ 
of the vector space $E_x$ i.e. $ds(T_0E_x)=T_xX$.

One says that the variety $X$ is \emph{elliptic} 
if it admits 
a spray $(E,p,s)$ which is dominating 
at each point $x\in X$, see \cite[3.5A]{Gr}. 
 It is immediate from the definition that 
every elliptic manifold is unirational. 
An algebraic variety $X$ is said to be 
\emph{uniformly rational} (or 
\emph{regular} in the terminology of 
\cite[3.5D]{Gr})
if every point $x\in X$ has a neighborhood in $X$
isomorphic 
to an open subset of $\A^n$. 
It is called \emph{stably uniformly rational}
 if $X\times\A^k$ is uniformly rational for some
 $k\ge 0$.
A complete stably uniformly rational 
variety is elliptic, see 
\cite[Theorem 1.3 and Corollary 3.7]{AKZ24}. 

It is shown in \cite[Example 2.4]{BB} that 
a rational smooth cubic hypersurface $X$ in 
$\PP^{n+1}$, $n\ge 2$, is
uniformly rational (cf. also \cite[Remark 3.5E$'''$]{Gr}). 
Hence $X$ is elliptic. 
There are examples of such hypersurfaces $X$
of any even dimension $n=2k\ge 2$. 

As another example, consider 
a nodal cubic threefold $X$ in $\PP^4$. It is well known 
that $X$ is rational. If $X$ has just a single node, say $P$, 
then the blowup
${\rm Bl}_P(X)$ of $X$ at the node is a smooth, 
uniformly rational (hence elliptic) threefold, 
see \cite[Proposition 3.1]{BB}. 
The same holds
for a small resolution $\tilde X\to X$
provided $X$ has several nodes and the resolution 
$\tilde X$ is an algebraic variety, 
see \cite[Theorem 3.10]{BB}. Such 
an algebraic resolution exists, for instance, if 
$X$ has exactly 6 nodes, see \cite[Proposition 3.8]{BB}. 
In the same spirit, one can construct
examples of uniformly rational small resolutions 
of nodal quartic double solids, see \cite{CZ24}. 

Yet another class of 
uniformly rational (and hence elliptic) projective 
varieties consists of 
smooth complete intersections of two quadrics in 
$\PP^{n+2}$, $n\ge 3$, see \cite[Example 2.5]{BB}.

Gromov asked in \cite[3.5B$''$]{Gr} whether 
any rational smooth projective
variety is elliptic; the answer is still unknown. 
Gromov also discussed in \cite[3.5B$''$]{Gr} 
a conjectural equivalence 
between ellipticity and unirationality. 
In the opposite direction, the question arises 
whether the ellipticity
implies the (stable, uniform) rationality. 

The answer to the latter question is negative.
 Indeed, for a certain natural number $n$ there exists 
a finite subgroup 
$F$ of $\SL(n,\CC)$ 
such that the quotient $Y=\SL(n,\CC)/F$ 
is stably irrational, see e.g. \cite[Theorem 3.6]{Sal84}, 
\cite{Bog87},
\cite[Example 1.22]{Po1} and \cite{Po2}.
Being an affine homogeneous 
manifold of a semisimple group,
$Y$ is flexible,
that is, the subgroup 
$\SAut(Y)\subset\Aut(Y)$
generated by all one-parameter 
unipotent subgroups of $\Aut(Y)$ 
acts highly transitively on $Y$, see 
\cite[Proposition 5.4]{AFKKZ13}. 
The latter implies that $Y$ is elliptic, 
see e.g. \cite[Proposition A.3]{AFKKZ13}. 
Thus, $Y$ is an example of a
stably irrational elliptic affine
manifold. 

Our aim is to find an  irrational smooth projective 
variety $X$ which is Gromov's elliptic. 
It is not certain that the above example can be explored 
for this purpose.
Indeed, given a smooth elliptic affine variety $Y$,
we lack the construction
of a smooth elliptic completion $X$ of
$Y$.

Nevertheless, 
we will establish the ellipticity of 
every smooth cubic threefold $X\subset\PP^4$, 
see Theorem \ref{mthm}.
By the celebrated Clemens-Griffiths 
theorem \cite{CG} such a threefold defined over $\CC$ 
is irrational. 
According to Murre \cite{Mur73}
the same holds for every smooth cubic threefold defined 
over an algebraically closed field of characteristic 
different from 2. 

Our main result is the following theorem.
\begin{theorem}\label{mthm} A smooth 
cubic hypersurface $X\subset\PP^{n+1}$ 
of dimension $n\ge 2$ is elliptic 
in the sense of Gromov. 
\end{theorem}
\begin{remarks} $\,$\\
{\rm
1. It is not known
whether there exists
a stably (uniformly)
rational smooth cubic threefold, see e.g. 
\cite{CT19}. 
Also,  no example of an
irrational smooth cubic hypersurface 
of dimension $\ge 4$ is known.
Recall that every smooth cubic hypersurface 
of dimension 
$\ge 2$ is unirational, 
see \cite{Kol02} or \cite[Remark 5.13]{Hu}.

2. A smooth cubic threefold $X$ 
is far from being homogeneous. Indeed,
due to Matsumura--Monsky's theorem 
(see \cite{MM}) $\Aut(X)$ is a finite group. 

3. Notice that any rational smooth projective 
surface $X$ admits a covering  by open subsets 
isomorphic to $\A^2$. Hence 
it is uniformly rational and also elliptic.
In the sequel we deal with cubic hypersurfaces 
of dimension $\ge 3$.
}
\end{remarks}
Recall the following open question.
\begin{question}[{\rm \cite[Remark 3 3.5E$'''$]{Gr}}] 
Does the Gromov ellipticity of 
a smooth complete variety survive
the birational maps generated 
by blowups and 
contractions with smooth centers? 
\end{question}
There are some partial results 
concerning this question. 
The uniform rationality is 
preserved under blowups 
with smooth centers, 
see \cite[Proposition 2.6]{BB} 
(cf. \cite[3.5E-E$''$]{Gr}).  
Hence, also the ellipticity 
of a complete uniformly 
rational variety is preserved, 
see \cite[Corollary 1.5]{AKZ24}. 
For not necessary complete 
or uniformly rational elliptic varieties, 
the ellipticity is preserved under blowups 
with smooth centers
under some additional assumptions,
see \cite[Corollary 3.5.D$''$]{Gr}, 
\cite[Corollary 2]{LT17}, \cite{KKT18} 
and  \cite[Theorem 0.1]{KZ2}. 
The preservation of 
ellipticity under blowdowns with smooth 
centers is unknown. See also 
the discussion in \cite{Zai24}.

Forstneri\v{c} \cite{For17b} 
proved that  every elliptic  projective manifold $X$ 
of dimension $n$
admits a surjective morphism from $\CC^n$, 
whose restriction 
to an open subset of $\CC^n$
is smooth and also surjective. Due to Kusakabe \cite{Kus22}, 
this remains true for 
not necessary complete elliptic algebraic manifold $X$
over $\kk$, provided one replaces $\CC^n$ by 
$\A^{n+1}_{\kk}$. 

Recall that 
a \emph{generalized affine cone} $\hat Y$ 
over a smooth projective variety $X$ defined 
by an ample $\QQ$-divisor $D$ on 
$X$ is the affine variety 
\[\hat Y=\Spec\left(\bigoplus_{k=0}^\infty 
H^0\left(X,\cO_X(\lfloor kD\rfloor)\right)\right),\]
see, e.g., \cite[Sec. 1.15]{KPZ13}. 
In the case where  $D$ is a hyperplane section of 
$X\subset \PP^n$ the cone $\hat Y$ 
is the usual affine cone over $X$. 
The corresponding punctured cone $Y$ 
over $X$ is obtained from $\hat Y$
by removing the vertex of $\hat Y$; so $Y$ 
is a smooth quasiaffine variety.

Exploring the above results together with
\cite[Corollary 3.8 and Proposition 6.1]{KZ2}  
and Theorem \ref{mthm},
we deduce the following immediate corollary. 
\begin{corollary} Let $X\subset\PP^{n+1}$,
$n\ge 2$ 
be a smooth cubic hypersurface and $Y$ be a 
punctured generalized affine cone over $X$ defined 
by an ample polarization of $X$. Then the following hold.
\begin{itemize}
\item
$Y$ is elliptic in Gromov's sense.  
\item
There exist surjective morphisms $\A^{n+1}\to X$ resp.
$\A^{n+2}\to Y$ which are smooth and surjective
on appropriate open subsets $U\subset\A^{n+1}$ 
resp. $V\subset\A^{n+2}$.
\item If $\kk=\CC$, then there exists 
a surjective morphism $\A^n\to X$
 which is smooth and surjective
on an appropriate open subset $U\subset\A^n$.
\item The monoid of endomorphisms $\End(Y)$ 
acts $m$-transitively on $Y$ for every natural number $m$. 
\end{itemize}
\end{corollary}
\section{Criteria of ellipticity }
For the proofs of the following proposition see
\cite[3.5B]{Gr}, 
\cite[Propositions 6.4.1 and 6.4.2]{For17a}, 
\cite[Remark 3]{LT17} 
and \cite[Appendix B]{KZ2}.
\begin{proposition}[Gromov's extension lemma]\label{lem:ext}
Let $X$ be a smooth complete variety and $(E,p,s)$ 
be a spray on an open subset $U\subset X$
with values in $X$, that is, $p\colon E\to U$ 
is a vector bundle on $U$ with zero section $Z$
and $s\colon E\to X$ 
is a morphism such that $s|_Z=p|_Z$. Then for each 
$x\in U$ there is a smaller open neighborhood $V\subset U$
of $x$ in $X$ and a spray $(E',p',s')$ on $X$ such that 
$(E',p',s')|_V=(E,p,s)|_V$.
\end{proposition}
Recall that a variety $X$ is said to be \emph{locally elliptic} 
if for any $x\in X$ there is a local spray 
$(E_x,p_x,s_x)$ defined on 
a neighborhood $U$ of $x$ in $X$ and  dominating at $x$ 
such that $s_x\colon E\to X$ takes values in $X$. 
The variety $X$ is called 
\emph{subelliptic} if there is a family of sprays  
$\{(E_i,p_i,s_i)\}_{i\in I}$ 
on $X$ which is dominating at each 
point $x\in X$, that is,
\[T_xX=
{\rm span}({\rm d}s_i (T_{0_{i,x}} E_{i,x})\,|\,i\in I)
\quad \forall x\in X.\]
The following corollary on 
the equivalence of local and global ellipticity
 is given in \cite[3.5B$'$]{Gr}; 
see also \cite[Theorem 1.1 and Corollary 2.4]{KZ1}.
\begin{corollary}[Gromov's Localization Lemma]
\label{rem:Gr} 
The ellipticity of a smooth algebraic variety $X$ is equivalent 
to its local ellipticity. 
\end{corollary}
\begin{proof}[Sketch of the proof] 
Obviously, ellipticity implies local ellipticity.
Suppose that $X$ is locally elliptic. 
Choose a finite covering of $X$ that is
equipped with sprays with values in $X$ 
dominating on the elements 
of this covering. We can assume that 
the vector bundles 
of these sprays are trivial. Thus, they can be 
decomposed into a 
direct sum of trivial line bundles. 
By Proposition \ref{lem:ext} one can extend 
the resulting rank 1 
sprays to sprays defined 
on the whole $X$,
and then compose them. 
For rank 1 sprays the composition is again a spray, 
see \cite[Proposition~2.1]{KZ1},
which is dominating in our case. 
\end{proof}
\begin{remark}
The composition of sprays is defined in \cite[1.3.B]{Gr}
for sprays of arbitrary ranks. 
It is mentioned  in \cite[Section~1.3]{Gr}
that this does not give a spray, in general, if the ranks 
of participating sprays
are $>1$. Note, however, that this 
circumstance is omitted in the proof of 
Lemma 3.5B in  \cite{Gr}.
\end{remark}
Likewise, we obtain the following result.
\begin{proposition}[\rm{\cite[Theorem 1.1]{KZ1}}] 
\label{prop:KZ}
The ellipticity of a smooth variety is equivalent 
to its subellipticity. 
\end{proposition}
Given a spray $(E,p,s)$ with values in $X$ and 
a point $x\in X$, 
the constructible subset $O_x:=s(E_x)\subset X$ 
is called the \emph{$s$-orbit} of $x$.
The proof of Proposition \ref{prop:KZ}
leads to the following lemma, cf. \cite{KZ1}.
\begin{lemma}\label{cor:crit-ell-2}
A smooth variety $X$ of dimension $n$ 
is elliptic if for every $x\in X$ 
there exist local rank 1 sprays
$(E_i,p_i,s_i)$, $i=1,\ldots,n$ with values in $X$ 
defined on respective 
neighborhoods $U_i$ of $x$ 
such that the 
$s_i$-orbits $O_{i,x}$ of $x$ are curves with 
 local parameterizations 
$s\colon (E_{i,x},0)\to (O_{i,x},x)$ \'etale at $x$
whose tangent  vectors at $x$ span $T_xX$. 
\end{lemma}
\begin{proof}[Sketch of the proof] Shrinking $U_i$ if necessary
we may suppose 
that $(E_i,p_i,s_i)$ is defined on the whole 
$X$ for $i=1,\ldots,n$, 
see  Proposition \ref{lem:ext}. 
Since the tangent lines $ds(T_0E_{i,x})$  at $x$ 
to the $s_i$-orbits 
$O_{i,x}$, $i=1,\ldots,n$
span $T_xX$,
the composition of these extended sprays is a spray 
of rank $n$ on $X$ dominating at $x$. 
This domination spreads to a neighborhood 
of $x$ in $X$.
It follows that $X$ is locally elliptic (and also subelliptic). 
By Gromov's Localization Lemma,
see Corollary \ref{rem:Gr} (or alternatively 
by Proposition 
\ref{prop:KZ})
$X$ is elliptic. 
\end{proof}
In our proof of Theorem \ref{mthm} in the next section 
we construct on any smooth cubic hypersurface $X$ 
of dimension $n\ge 3$ an $n$-tuple of 
independent rank 1 sprays 
$(E_i,p_i,s_i)$ dominating in the orbit directions. 
\section{Ellipticity of cubic hypersurfaces}\label{sec:proof}
We use the following notation. 
\begin{notation}
For a subset $M\subset\PP^n$ 
we let $\langle M\rangle$ 
be the smallest projective subspace 
which contains $M$. 
Given a smooth hypersurface $X\subset\PP^{n+1}$ 
and a point $x\in X$ we let $S_x=\T_xX\cap X$, 
where $\T_xX$ stands for the projective tangent space 
of $X$ at $x$. Let $\cC_x$  
stand for the union of lines 
on $X$ through the point $x\in X$. 
Clearly, $\cC_x\subset S_x$. 
Recall that for $n\ge 3$ every smooth cubic hypersurface 
$X\subset\PP^{n+1}$
is covered by projective lines 
(this follows, for instance, from \cite[Corollary 8.2]{CG}). 
Therefore, $\cC_x$ has positive dimension for every 
$x\in X$. 

Given a cubic threefold $X\subset\PP^{4}$, 
the equality $\cC_x=S_x$
holds if and only if 
$x$ is an Eckardt point of $X$. 
Recall that a point $x\in X$ is called an 
\emph{Eckardt point} if there is an  infinite number 
of lines on $X$ passing through $x$. In the latter case 
$\cC_x$ is the cone with vertex $x$
over a plane elliptic cubic curve. 
Notice that a general cubic threefold 
has no Eckardt points, 
a cubic threefold can contain at most 
30 Eckardt points (see \cite[p. 315]{CG}), 
and the maximal number of 30 Eckardt points
is attained only for the Fermat cubic threefold, 
see \cite[p. 315]{CG}, \cite{Rou09} 
and \cite[Chapter 5, Remark 1.7]{Hu}. 
\end{notation}
\begin{lemma}\label{lem:n-lines}
Given a smooth cubic hypersurface 
$X\subset\PP^{n+1}$ where $n\ge 3$,  
the tangent space 
$T_xX$ at a  general point 
$x\in X$ is spanned by some $n$ 
tangent lines to 
projective lines on 
$X$ passing through $x$.
\end{lemma}
\begin{proof}
Let us start with the case $n=3$, that is,
 let $X$ be a smooth cubic threefold.
Through a general point $x\in X$ pass 
exactly 6 lines, see e.g. 
\cite[Proposition (1.7)]{AK77}
or \cite[Chapter 5, Exercise 1.4]{Hu}. 
More precisely, there is a proper closed subset
$Y\subset X$ swept out by 
the ``lines of the second type'' on $X$, and through every 
point $x\in X\setminus Y$
pass exactly 6 distinct lines on $X$, 
see \cite[Lemma (1.19)]{Mur72}.
These 6 lines form the cone $\cC_x$. 
No four of them are coplanar, 
hence the tangent space $T_xX$
is spanned by the tangent lines to
some three of these projective lines through $x$.

Let us show now that the latter property 
holds as well in higher dimensions.
Given a smooth cubic hypersurface 
$X\subset\PP^{n+1}$, $n\ge 4$ and a point
$x\in X$, consider the projective variety 
\[\PP\cC_x:=\{T_xl\,|\, x\in l\subset \cC_x\}\subset 
\PP T_xX\]
where $l$ stands for a line on $X$. 
Clearly, $T_xX$ is spanned 
by the tangent lines $T_xl$ 
for $l\subset \cC_x$ if and only if $\PP\cC_x$ 
is linearly nondegenerate in 
$\PP T_xX\cong\PP^{n-1}$.
There exists an open dense subset $U\subset X$
such that for every $x\in U$ the projective cone 
$\cC_x$ has
codimension 2 in $\T_xX$. The dimension 
$\dim\,\langle \PP\cC_x\rangle$ 
is lower semicontinuous on 
$U$ and attains its
maximal value, say $m$
on an open dense subset $U_0\subset U$. 
Let $\mathcal{V}\subset\PP TX|_{U_0}$ 
be the subvariety swept out by the 
$\langle \PP\cC_x\rangle$ for $x\in U_0$. 

Suppose that, contrary 
to the assertion of the lemma,   $m<n-1$, so that
$\mathcal{V}$ is a proper closed subset of 
$\PP TX|_{U_0}$. Choose a general linear section 
 $Y$ of $X$ by a subspace $\PP^4\subset\PP^{n+1}$. 
Then $Y$ is a smooth cubic threefold. 
Furthermore, for a general $x\in Y\cap U_0$
the proper subspaces $\PP T_xY\cong\PP^2$ 
and $\langle \PP\cC_x\rangle\cong\PP^m$ of 
$\PP T_xX\cong\PP^{n-1}$ are transversal. 
On the other hand, 
$T_xY$ is spanned 
by the tangent lines to the lines on $Y$ 
passing through $x$.
Thus, we have $\PP T_xY\subset \langle \PP\cC_x\rangle$. 
This contradiction ends the proof. 
\end{proof}
\begin{notation}
For a general point $u\in X$ we let $S^*_u$ be the set of 
points $x\in X$ such that $u\in \T_xX$. Let $x\in S^*_u$ 
be a  general point
and $u^*=\T_uX$ resp. $x^*=\T_xX$
be the corresponding points of the dual hypersurface 
$X^*\subset (\PP^{n+1})^{\vee}$.
Then $u\in \T_xX$ if and only if 
$x^*\in \T_{u^*}X^*$, that is, 
$x\in S^*_u$  if and only if  $x^*\in S_{u^*}$. 
It follows that
$S^*_u$ is a hypersurface in $X$ 
passing through $u$. 
It is easily seen that $\cC_u\subset S_u^*$, 
and so $\cC_u\subset S_u\cap S_u^*$.
\end{notation}
\begin{lemma}\label{lem:bitangent}
For a smooth cubic hypersurface 
$X\subset \PP^{n+1}$, $n\ge 1$
one has $S_u\cap S_u^*=\cC_u$ for every $u\in X$.
\end{lemma}
\begin{proof} 
If $x\in S_u\cap S_u^*$ is different from $u$ then 
$l:=\langle x,u\rangle\subset \T_xX\cap \T_uX$. 
By the B\'ezout theorem,
the bitangent line $l$ of $X$
is contained in $X$.
It follows that $x\in \cC_u$. Hence
we have $S_u\cap S_u^*\subset \cC_u$. 
Since also $\cC_u\subset S_u\cap S_u^*$ one has
$S_u\cap S_u^*= \cC_u$.
\end{proof}

Let again $X$ be a smooth cubic hypersurface 
in $\PP^{n+1}$, $n\ge 1$. 
Following \cite[Example 2.4]{BB} and \cite{BKK} 
let us consider the birational 
self-map $\tau_u\colon X\dasharrow X$
which sends a general point $x\in X$ 
to the third point $y=\tau(x)$ 
of intersection of the line
$\langle x,u\rangle$ with $X$. 
Thus, $X$ cuts out on the line
$\langle x,u\rangle$ the reduced divisor $x+u+y$. 
In fact, $\tau_u$ is well defined unless 
$x=u$ or $x\neq u$ and the line $\langle x,u\rangle$ 
is contained in $X$.
Thus, $\tau_u\colon X\dasharrow X$ is regular on 
$X\setminus \cC_u$. 
\begin{lemma}\label{lem:fixed-pts}
Let $X\subset \PP^{n+1}$, $n\ge 1$ 
be a smooth cubic hypersurface 
 and $x,u\in X$ be such that
$x\in X\setminus \cC_u$.
Then $\tau_u(x)=\tau_u(x')$ for a point $x'\in X$
if and only if either 
$u\in \T_xX$ (i.e. $x\in S_u^*$) and $x'=x$,  or
$x\in \T_uX$ (i.e. $x\in S_u$) and $x'=u$. 
\end{lemma}
\begin{proof} Notice that by our assumption $x\neq u$
and $\langle x,u\rangle\not\subset X$. 
Suppose that 
$\tau_u(x)=\tau_u(x')$ holds. Then we have 
$\langle x,u\rangle=\langle x',u\rangle$. So
the points $x,x',u$ are alined. If $x'\notin\{x,u\}$, 
then $\tau_u(x)=x'$ and $\tau_u(x')=x$. 
Applying $\tau_u$ once again
we obtain
$x'=\tau_u(x)=\tau_u(x')=x$, a contradiction. 
Thus, $x'\in\{x,u\}$, and so the divisor  cut out 
by $X$ on the line $\langle x,u\rangle$ is
either $2x+u$, or $2u+x$. 
We have $x'=x$ in the former case 
and $x'=u$ in the latter case. 
Now the assertion follows. 
\end{proof}
The following corollary is immediate. 
\begin{corollary}\label{cor:involution} $\,$
\begin{itemize}
\item[(a)]
The morphism 
$\tau_u\colon X\setminus \cC_u\to X$ 
contracts the hyperplane section 
$S_u\setminus \cC_u$ 
to the point $u$ and fixes 
pointwise the variety
$S_u^*\setminus S_u=
S_u^*\setminus \cC_u$
(see Lemma \ref{lem:bitangent}). 
\item[(b)]
The  indeterminacy points
of $\tau_u$ are contained in $\cC_u$. 
\item[(c)]
The affine threefold $X\setminus 
S_u\subset X\setminus \cC_u$ 
is invariant under $\tau_u$ and 
$\tau_u|_{X\setminus S_u}$ is 
a biregular involution with mirror 
$S_u^*\setminus \cC_u$ 
that acts freely on $X\setminus 
(S_u\cup S_u^*)$.
\end{itemize}
\end{corollary}
In the proof of Theorem \ref{mthm} 
we use the following 
construction of a rank 1 spray on a 
smooth cubic hypersurface.

\medskip

\begin{proposition} \label{lem:1-spray}
Let $X\subset\PP^{n+1}$, $n\ge 3$ 
be a smooth cubic hypersurface, $y\in X$ 
be an arbitrary point and 
$x\in X$ be a general point. 
Let $u\in X\cap\langle x,y\rangle$ be a point
different from $x$ and $y$. Choose a line
$l\subset \cC_x$ on $X$ through $x$,
and let $C^*=\tau_u(l\setminus S_u)\cong\A^1$. 
Then there exists 
a rank $1$ spray $(E,p,s)$ on $X$ such that the 
$s$-orbit $O_y$ of $y$
coincides with $C^*$ and $ds(T_0E_y)$ 
is the tangent 
line to $O_y=C^*$  at $y$.
\end{proposition}
\begin{proof}
Since $X$ cuts out on 
the line $\langle x,y\rangle$ the reduced divisor 
$x+u+y$ we have 
\[x,y\in U_u:=X\setminus 
(S_u\cup S_u^*)\quad\text{and}\quad 
u,y\in U_x:=X\setminus  (S_x\cup S_x^*),\]
see Lemma \ref{lem:fixed-pts}. In particular, 
$x\notin \T_uX$. Hence the line $l\subset X$ 
passing through $x$ meets $\T_uX$
in a single point, say $z\in S_u=\T_uX\cap X$. 

Fix an isomorphism 
$f\colon \PP^1\stackrel{\cong}\longrightarrow l$
that sends $0$ to $x$ and $\infty$ to $z$. 
Then $f$ embeds
$\A^1=\PP^1\setminus\{\infty\}$ onto 
$l\setminus \{z\}\subset  X\setminus S_u$. 
Since $\tau_u|_{X\setminus S_u}
\in\Aut(X\setminus S_u)$, see 
Corollary \ref{cor:involution}(c),
the map 
$\phi_u:=\tau_u\circ f|_{\A^1}\colon \A^1\to C^*$
is an isomorphism.

By Corollary \ref{cor:involution}(c) $\tau_x|_{U_x}$ is a 
biregular involution acting freely on $U_x$ 
and interchanging $u$ and $y$. 
Since the projective line $\langle x,u\rangle$ 
is not tangent to $X$,
we can choose an open neighborhood 
$V_u\subset U_x$ of $u$ such that also the line 
$\langle x,u'\rangle$ 
is not tangent to $X$ for each $u'\in V_u$. 
Letting $y'=\tau_{u'}(x)\in \langle x,u'\rangle\cap X$ 
we have $\tau_x(u')=y'$.
Letting  $V_y=\tau_x(V_u)\subset U_x$
 the restriction $\tau_x|_{V_u}\colon 
V_u\stackrel{\cong}
\longrightarrow V_y$ is biregular, and so
$V_y$ is a neighborhood of $y=\tau_u(x)$ in $X$. 

Shrinking $V_u$ if necessary we may assume that  
for any $u'\in V_u$ the tangent hyperplane 
$\T_{u'}X$ does not pass through $x$, so that
$z(u'):=l\cap\T_{u'}X\neq x$. 

Consider the trivial $\PP^1$-bundle 
$\pi\colon V_u\times l\to V_u$ 
over $V_u$ along with the constant section 
$Z_0=V_u\times \{x\}$ of $\pi$.  
The subset $\{z(u')|u'\in V_u\}\subset V_u\times l$ 
defines a section, say $Z'$ of
$\pi$ disjoint with $Z_0$. 
The restriction $\pi\colon (V_u\times l)\setminus Z'\to V_u$
is a smooth $\A^1$-fibration with irreducible fibers. 

There is an automorphism of the $\PP^1$-bundle 
$\pi\colon V_u\times l\to V_u$ 
identical on the base $V_u$
that preserves $Z_0$ and sends $Z'$ to 
a constant section, say, 
$Z_\infty$ disjoint with $Z_0$. 
This yields a trivialization 
\[F\colon V_u\times\A^1\stackrel{\cong|_{V_u}}{\longrightarrow}
(V_u\times l)\setminus Z'\] 
that extends $f$ and sends the zero section of 
$V_u\times\A^1\to V_u$  to $Z_0$.

Consider the morphism
\[\phi\colon V_u\times \A^1\to X,
\quad (u',t)\mapsto\phi_{u'}(t):=\tau_{u'}(F(t)).\]
Let 
\[s=\phi\circ (\tau_x\times\id_{\A^1})\colon 
V_y\times\A^1\to X,
\,\,\, (y',t)\mapsto \phi_{u'}(t)\,\,
\text{where}
\,\,u'=\tau_x(y')\in V_u.\]
Consider also the trivial line bundle 
$p\colon E=V_y\times\A^1\to V_y$, 
where $p$ is the first projection, 
with zero section $Z=V_y\times\{0\}$. 
We have $s(y',0)=y'$, that is $s|_{Z}=p|_{Z}$. Thus, 
the triplet $(E,p,s)$ 
is a spray of rank 1 on $V_y$ 
with values in $X$. 
The $s$-orbit $O_{y}$ of $y$
coincides with $\phi_{u}(\A^1)=C^*$ and
the map 
$ds|_{T_0E_{y}}\colon T_0E_{y}\to T_{y}O_{y}$
is onto. 
Due to Proposition \ref{lem:ext} $(E,p,s)$ 
can be extended to a spray on $X$.
\end{proof}
\begin{remark} {\rm We claim that 
the general $s$-orbits $O_{y'}$ in $X$
are smooth affine conics whose closures
are fibers of the conic bundle $\hat X\to\PP^2$ 
resulting 
from the blowup of $\hat X\to X$ with center $l$. 
Indeed, let $y'\in V_y$ be a general point.
Then $\langle x,y'\rangle$ being a general 
projective line 
in $\PP^4$ through $y'$,
$X$ cuts the plane
$L':=\langle l, y'\rangle$ along the line $l$ 
passing through $x$
and the residual smooth conic $C_{y'}$ 
passing through $u'$ and $y'$. 
So, the plane $L'$ is 
$\tau_{u'}$-invariant and $\tau_{u'}$
interchanges $l$ and $C_{y'}$ fixing
their intersection points that are
points
of $l\cap S_{u'}^*$, see 
Corollary \ref{cor:involution}(c). 
Furthermore, $C_{y'}$
is the closure of  the $s$-orbit
$O_{y'}=\tau_{u'}(l\setminus\{z(u')\})$. 
}
\end{remark}
We are now ready to prove  Theorem \ref{mthm}.
\begin{proof}[Proof of Theorem \ref{mthm}]
Since every complete smooth rational surface 
is elliptic, 
see \cite[Theorem 1.1]{KZ2},
we may suppose that $n\ge3$.  
Due to Corollary \ref{rem:Gr}  
it suffices to show that $X$ is locally elliptic. 
Fixing an arbitrary $y\in X$,
choose a general point $x\in X$ 
and $n$ projective lines $l_1,\ldots,l_n$ on $X$ 
through $x$ such that
their tangent lines $T_xl_i$ span $T_xX$, 
see Lemma \ref{lem:n-lines}. 
Applying Proposition \ref{lem:1-spray} 
to each of the lines $l_i$
yields rank 1 sprays $(E_i,p_i,s_i)$, 
$i=1,\ldots,n$, on $X$ 
dominating at $y$ along their respective 
orbits $O_{i,y}$.
We claim that the collection of sprays 
$\{(E_i,p_i,s_i)\}_{i=1,\ldots,n}$ is dominating 
at $y$. The latter implies that $X$ is subelliptic, 
and so elliptic, 
see Corollary \ref{rem:Gr}  (cf. also 
Lemma \ref{cor:crit-ell-2}). 
The above domination is equivalent to the fact
that the tangent lines $T_yO_{i,y}$ span 
the tangent space $T_yX$. By the construction 
of Proposition \ref{lem:1-spray} 
we have 
\[T_yO_{i,y}=d\phi_{i,u}(T_xl_i)\,\,\,\text{where}
\,\,\,\phi_i=\tau_u\circ f_i\,\,\,\text{with}
\,\,\,
f_i\colon (\A^1,0)\stackrel{\cong}
\longrightarrow (l_i\setminus \{z_i\},x).\]
Therefore, 
\[{\rm span}(T_yO_{1,y},\ldots,\,T_yO_{n,y})
=d\tau_u({\rm span}(T_xl_1,\ldots,\,T_xl_n))=
d\tau_u(T_xX)=T_yX,\]
which proves our claim. The domination at $y$ 
spreads to a 
neighborhood of $y$ in $X$, which gives 
the local ellipticity (and the subellipticity), 
hence also the ellipticity. 
\end{proof}
\section{An alternative proof}
We suggest here an alternative proof of 
Theorem \ref{mthm}. The
geometrical arguments
used in Section \ref{sec:proof}
are now replaced by references to
Forstneri\v{c}-Kusakabe's theorem.
\begin{lemma}\label{p1} Let $X$ be a variety
and $S\subset X\times\PP^n$ be 
a subvariety of codimension at least $2$. 
Then for a general point $(x,a)\in X\times\PP^n$ 
and a general projective line 
$L\subset\PP^n$ through $a$
there exists a neighborhood $U_x$ of $x$ in $X$
such that $(U_x\times L)\cap S=\emptyset$. 
\end{lemma}

\begin{proof}
 It suffices to show that 
$(\{x\}\times L)\cap S=\emptyset$ 
for a general choice of 
$(x,a)\in X\times\PP^n$
 and $L$ in the Grassmannian ${\bf G}_a(1,n)$ 
 of lines in $\PP^n$ 
 passing through $a$. 
 Since $(x,a)\in X\times\PP^n$ is general 
 we have $(x,a)\notin S$.
Let $L_v$ be a line passing  through $a$ 
in direction of a general vector
$v\in\PP T_a\PP^n\simeq\PP^{n-1}$. 
Fixing $a$ assume to the contrary that 
$(\{x\}\times L_v)\cap S\neq\emptyset$ 
for general 
$x\in X$ and $v\in \PP^{n-1}$.
Then $\dim(S)\ge \dim(X)+n-1
=\dim(X\times\PP^n)-1$. 
This contradicts the assumption that 
${\rm codim}_{X\times\PP^n}(S)\ge 2$. 
\end{proof}
\begin{corollary}\label{c1} 
Let $X$ be a normal algebraic variety,
$\psi \colon X\times \A^n\dashrightarrow X$ 
be a dominant rational map,
and $S$ be the set of indeterminacy points 
of $\psi$. 
Then for a general point $(u,a)\in X\times\A^n$ and 
a general affine line $l\subset \A^n$ passing through $a$
there is a neighborhood $U_u$ of $u$ in $X$
such that $(U_u\times l)\cap S=\emptyset$, and so
$\psi|_{U_u\times l}\colon U_u\times l\to X$ 
is a morphism. 
\end{corollary}
\begin{proof} It suffices to extend $\psi$ to a rational map 
$\bar\psi\colon X\times\PP^n\dasharrow X$, to replace $S$ 
by the indeterminacy set 
of $\bar\psi$ and to apply Lemma \ref{p1} in this new framework. 
\end{proof}
The next proposition is an analog of 
Proposition \ref{lem:1-spray}.
\begin{proposition}\label{p2} Let $X$ be a smooth 
cubic hypersurface in $\PP^{n+1}$ and $y$ 
be a point in $X$.
Then there is a neighborhood $U_y$ of $y$ 
in $X$ and a rank $1$spray $(E,p,s)$ on $U_y$ with values in $X$
such that $E=U_y\times \A^1$, $p\colon U_y\times \A^1\to U_y$
 is the first projection
and a morphism $s\colon U_y\times \A^1\to X$ 
verifies the following conditions:
\begin{itemize}
\item $s|_Z=p|_Z$  where $Z=U_y\times\{0\}$ is the 
zero section of $p\colon E\to U_y$;
\item the orbit map 
$s|_{E_y}\colon E_y\to s(E_y)$ is smooth 
at the origin $0_y\in E_y$;
\item $ds|_{T_{0_y}E_y}$ sends a tangent 
vector to $E_y$ at $0_y$ 
to a general vector in $T_yX$.
\end{itemize}
\end{proposition}
\begin{proof}
By Kusakabe's theorem, see \cite{Kus22}, 
there is a surjective morphism $f: \A^{n+1} \to X$. 
Let $g=(\id_X, f)\colon X\times \A^{n+1}\to X\times X$.
Define a rational map 
$\tau \colon X\times X\dashrightarrow X$ 
by letting 
$\tau(u,x)=\tau_u(x)$.
Then the composition 
\[\psi\colon X\times 
\A^{n+1}\stackrel{g}{\longrightarrow} 
X\times X\stackrel{\tau}{\dashrightarrow} X,
\qquad (u,a)\mapsto \tau_u(f(a)),\] 
is a dominant rational map.
Letting
$I(\tau)\subset X\times X$
and $S:=I(\psi)\subset X\times \A^{n+1}$ 
be the indeterminacy sets of $\tau$ and $\psi$,
 respectively,
we have $S\subset g^{-1}(I(\tau))$ and
${\rm codim}_{X\times \A^{n+1}}(S)\ge 2$.

Let now $L_y$ be a general line in $\PP^{n+1}$ 
through $y$. 
It meets $X$ at general points $x$ and $u$
different from $y$.
Recall that the indeterminacy set $I(\tau_u)\subset X$ 
is contained in the union $\cC_u$ of lines on $X$ 
passing through $u$, see 
the paragraph preceding Lemma \ref{lem:fixed-pts}. 
Since $L_y\not\subset X$, we have
 $x\notin I(\tau_u)$, and so $(u,x)\notin I(\tau)$.
Since $S\subset g^{-1}(I(\tau))$, for $a \in f^{-1} (x)$ 
we have $(u,a)\notin S$ and $\psi(u,a)=y$.
 
Fix a general affine line $l\subset\A^{n+1}$ 
through $a$. Let $U_u$ be a neighborhood  of $u$ in $X$
such that $\psi$ is regular on 
$U_u\times l\simeq U_u\times \A^1$, 
see Corollary \ref{c1}. 
Then
$U_y :=\tau (U_u\times \{x\})$ is a neighborhood 
of $y$ in $X$ and  
$\tau_x\colon U_u\stackrel{\simeq}{\longrightarrow} U_y$
is an isomorphism. Composing $\psi$ with $(\tau_x, \id)$ 
one gets a morphism $s\colon U_y\times\A^1 \to X$. 
This defines a desired rank 1 spray $(E,p,s)$ on $U_y$ 
with values in $X$. Indeed, $s(x, a)=y$ 
for any $a\in f^{-1}(x)$.
Furthermore, $s|_{E_y}=s|_{(x,l)}=\tau_u(f(l))$ 
is the $s$-orbit of $y$.
Since $x\in X$ is a general point, it is not 
a critical value of $f$. 
Hence for $a\in f^{-1}(x)$, the morphism
$g=(\id_X,f)\colon X\times\A^{n+1}\to X\times X$ 
is dominant at $(u,a)$.
Since $l\subset\A^{n+1}$ is a general line 
passing through $a$, 
it follows that
the differential $ds$ sends $T_{0_y}E_y$ to 
a general vector line in $T_yX$. 
\end{proof}
Due to Gromov's Localization Lemma (see Corollary \ref{rem:Gr})
to prove Theorem \ref{mthm} it remains to apply 
the following corollary.
\begin{corollary}
$X$ is locally elliptic.
\end{corollary}
\begin{proof}
Choose independent general vectors $v_1,\ldots,v_n$ in $T_yX$
and the corresponding local rank 1 sprays $(E_i,p_i,s_i)$ defined on 
a common neighborhood $U$ of $y$ in $X$ with values in $X$, 
see Proposition \ref{p2}. Extend these sprays 
to sprays on the whole $X$, see Proposition \ref{lem:ext}.
The resulting collection of sprays is dominating on $U$.
This guarantees the subellipticity of $X$ on $U$. 
Now the claim follows by Proposition \ref{prop:KZ}. 
\end{proof}

\vskip 0.1in

{\bf Acknowledgments.} It is our pleasure to thank  
Yuri Prokhorov for valuable discussions, references and 
comments on a preliminary version 
of this paper. 

%

\maketitle
\end{document}